\documentclass[11pt]{article}
\usepackage{amsmath,fullpage,amssymb,amsthm,hyperref,enumerate}
\usepackage{tikz,color}

\usetikzlibrary{shapes,arrows,decorations.pathreplacing}
\tikzstyle{pnt}=[draw,ellipse,fill,inner sep=1pt]

\newtheorem{theorem}{Theorem}[section]
\newtheorem{proposition}[theorem]{Proposition}

\newtheorem{corollary}[theorem]{Corollary}
  
\newtheorem{problem}[theorem]{Problem}  

\theoremstyle{definition}
\newtheorem{example}[theorem]{Example}

\theoremstyle{remark}

\newcommand\G{\mathcal G}

\newcommand\R{\mathcal R}
\newcommand\M{\mathcal M}
\newcommand\Cat{C}
\renewcommand\S{\mathcal S}
\DeclareMathOperator\ds{ds}
\DeclareMathOperator\da{da}
\newcommand\sz{s_0}

\newcommand{\E}{\mathbb{E}}
\newcommand{\V}{\mathbb{V}}
\renewcommand{\Pr}{\mathbb{P}}
\newcommand{\bij}{\phi}

\title{The degree of asymmetry of sequences}

\author{
Sergi Elizalde\thanks{Department of Mathematics, Dartmouth College, Hanover, NH 03755, USA. {\tt sergi.elizalde@dartmouth.edu}.}
\and
Emeric Deutsch\thanks{NYU Tandon School of Engineering, Brooklyn, NY 11201, USA.}
}

\date{}

\begin{document}

\maketitle

\begin{abstract}
We explore the notion of {\em degree of asymmetry} for integer sequences and related combinatorial objects. The degree of asymmetry is a new combinatorial statistic that measures how far an object is from being symmetric. We define this notion for compositions, words, matchings, binary trees and permutations, we find generating functions enumerating these objects with respect to their degree of asymmetry, and we describe the limiting distribution of this statistic in each case. 
\end{abstract}

\medskip



\section{Introduction}

It is common in mathematics to study objects that are invariant under certain symmetries. Examples of such objects in combinatorics include
self-conjugate partitions~\cite[Prop.\ 1.8.4]{EC1}, symmetric plane partitions~\cite{Mac,Stanley}, symmetric planar maps~\cite{Albenque}, 
centrally symmetric dissections of a polygon~\cite{Simion}, palindromic compositions~\cite{HB}, symmetric lattice paths~\cite{DDS}, involutions, and centrosymmetric permutations.

To interpolate between the set of all objects and the subset of symmetric ones, it is natural to introduce the notion of {\em degree of asymmetry}, which is a measure of how asymmetric the object is, i.e., how far it is from being invariant under the symmetry operation. In particular, symmetric objects are those having degree of asymmetry equal to $0$.
To the best of our knowledge, the notion of degree of asymmetry is new. It has some similarities to other statistics studied in the literature, such as the number of {\em centered tunnels} in Dyck paths (introduced in~\cite{Elifpexc,EliDeu,EliPak}) or the number of transpositions of a permutation.

In this paper we study the degree of asymmetry of integer sequences, including compositions and words (Section~\ref{sec:comp_words}), binary words with the same number of zeros and ones (Section~\ref{sec:restricted}), and other related combinatorial objects such as binary trees (Section~\ref{sec:trees}),
matchings (Section~\ref{sec:matchings}) and permutations (Section~\ref{sec:perm}). We provide generating functions for these objects with a variable $t$ that marks their degree of asymmetry, so that the specialization $t=0$ enumerates in each case the subclass of symmetric objects.
We use singularity analysis of bivariate generating functions to determine the limiting distribution of the degree of asymmetry on these objects. For compositions and words, the resulting distributions are asymptotically normal, whereas in the case of binary trees and matchings we obtain convergence to discrete limit laws. For binary words with the same number of zeros and ones, we find a bijection that proves a surprising connection between statistics that relate to the degree of symmetry and statistics that count pattern occurrences.
This work is complemented by a concurrent paper by the first author~\cite{Elids}, which focuses on a similar notion of symmetry for lattice paths and partitions.

\section{Compositions and words}\label{sec:comp_words}

Let $S=(a_1,a_2,\dots,a_m)$ be a finite sequence. We define the {\em degree of asymmetry} of $S$ to be the number of pairs of symmetrically positioned distinct entries, and denote it by $\da(S)$. Specifically,
$$\da(S)=\left|\{i:1\le i\le m/2: a_i\neq a_{m+1-i}\}\right|.$$
For example, if $S=(2,5,1,3,2)$, then $\da(S)=1$. By definition, $\da(S)=0$ if and only if $S$ is palindromic, that is, equal to its reversal.

\subsection{Compositions}\label{sec:comp}

A sequence of positive integers $(a_1,a_2,\dots,a_m)$, where $m\ge0$ and $a_1+a_2+\dots+a_m=n$, is called a {\em composition} of $n$ with $m$ parts. 
Let $K(t,x,z)$ be the generating function for compositions where $t$ marks the degree of asymmetry, $x$ marks the number of parts, and $z$ marks the sum of its parts. 

\begin{proposition} \label{prop:comp}
$$K(t,x,z)=\frac{(1-z+xz)(1-z^2)}{(1-z)(1-(1+x^2)z^2)-2tx^2z^3}.$$
\end{proposition}

\begin{proof}
The generating function $K=K(t,x,z)$ satisfies 
$$K=1+x\sum_{i\ge1}z^i+tx^2\sum_{i,j\ge1}z^{i+j}K+(1-t)x^2\sum_{i\ge1}z^{2i}K,$$
where the summands correspond to the empty word, a word $i$ of length 1, a word of length at least $2$ starting with $i$ and ending with $j$, and the corresponding correction for words where $i=j$, which should be weighted by $t$.
It follows that $$K=1+\frac{xz}{1-z}+\frac{tx^2z^2K}{(1-z)^2}+\frac{(1-t)x^2z^2K}{1-z^2}.$$ 
Solving for $K$ gives the stated expression.
\end{proof}

More generally, one can consider compositions whose parts belong to a (finite or infinite) set $\R$,
and let $R(z)=\sum_{r\in \R} x^r$. Let $K_\R(t,x,z)$ be the generating function for such compositions with respect to the degree of asymmetry.

\begin{proposition} \label{prop:CR} 
$$K_\R(t,x,z)=\frac{1+xR(z)}{1-x^2R(z^2)-tx^2\left(R(z)^2-R(z^2)\right)}.$$
\end{proposition}

\begin{proof}
We claim that the generating function $K_\R=K_\R(t,x,z)$ satisfies
$$K_\R=1+xR(z)+x^2R(z^2)K_\R+tx^2(R(z)^2-R(z^2))K_\R.$$
Indeed, $x^2R(z^2)K_\R$ is the generating function for compositions with at least 2 parts where the first and the last entry are the same, and 
$tx^2(R(z)^2-R(z^2))K_\R(t,z)$ counts compositions where the first and the last entry are different.
Solving for $K_\R$ gives the stated expression.
\end{proof}

As expected, the substitution $t=1$ gives $K_\R(1,x,z)=1/(1-xR(z))$, the generating function for compositions with parts in $\R$. The substitution $t=0$ gives $$K_\R(0,x,z)=\frac{1+xR(z)}{1-x^2R(z^2)},$$ which is the generating function for palindromic compositions with parts in $\R$. This generating function has been obtained by Hoggatt and Bicknell~\cite[Thms.\ 1.1 and 1.2]{HB}.

Next we analyze the generating function 
\begin{equation}\label{eq:Ctz} 
K(t,z):=K(t,1,z)=\frac{1-z^2}{(1-z)(1-2z^2)-2tz^3}
\end{equation}
 from Proposition~\ref{prop:comp} in order to describe the asymptotic behavior of the degree of asymmetry of compositions. 
If $X_n$ is the random variable that gives the degree of asymmetry of a uniformly random composition of $n$, then
\begin{equation}\label{eq:Xn}
\Pr(X_n=k)=\frac{[t^kz^n]K(t,z)}{[z^n]K(1,z)}.
\end{equation}
Denote by $K_t$ and $K_{tt}$ the first and second partial derivatives of $K(t,z)$ with respect to $t$, respectively, and let $[z^n]f(z)$ denote the coefficient of $z^n$ in a formal power series $f(z)$. It is well known that the expected value and the variance of $X_n$ are given by
\begin{equation}\label{eq:EV}
\E X_n=\frac{[z^n]K_t(1,z)}{[z^n]K(1,z)} \quad\text{and}\quad
\V X_n=\frac{[z^n]K_{tt}(1,z)+[z^n]K_t(1,z)}{[z^n]K(1,z)}-\left(\frac{[z^n]K_t(1,z)}{[z^n]K(1,z)}\right)^2,
\end{equation}
respectively.

\begin{corollary}\label{cor:limitCtz}
The sequence of random variables $X_n$, giving the degree of asymmetry of a uniformly random composition of $n$, converges to a normal distribution with expected value and variance given by
$$\E X_n=\frac{n}{6}+O(1),\qquad \V X_n=\frac{5n}{108}+O(1). $$
\end{corollary}

\begin{proof} 
With $K(t,z)$ given by Equation~\eqref{eq:Ctz}, we can compute easily the expectation and variance using Equation~\eqref{eq:EV},
and the fact that $[z^n]K(1,z)=2^{n-1}$ for $n\ge1$,
$$[z^n]K_t(1,z)=\left(\frac{n}{12}-\frac{1}{19}\right)2^n-\frac{4}{9}(-1)^n$$
for $n\ge2$, and 
$$[z^n]K_{tt}(1,z)=\left(\frac{n^2}{72}-\frac{17n}{216}+\frac{2}{27}\right)2^n+\frac{16n-56}{27}(-1)^n$$
for $n\ge3$.

Convergence to a normal distribution follows from \cite[Thm.\ IX.9]{FS} or \cite[Thm.\ 1]{Bender}, noting that for $t$ in a neighborhood of $1$, the singularity of $K(t,z)$ closest to the origin is a simple pole at $z=\rho(t)$, where $\rho(1)=1$, $\rho'(1)=-1/12$, and $\rho''(1)=2/27$.
\end{proof}

It is well known \cite[Prop.\ III.4]{FS} that the expected number of parts of a random composition of $n$ is $(n+1)/2$, and that the distribution is concentrated around this mean.  The value of $\E X_n$ given in Corollary~\ref{cor:limitCtz} tells us that, of the roughly $n/4$ pairs of symmetrically positioned entries in a typical composition of $n$ (for large $n$), an average of about $2/3$ consist of distinct entries, contributing to the degree of asymmetry, while the remaining $1/3$ consist of equal entries.

\subsection{Words}\label{sec:words}

Words over an $m$-ary alphabet 
can be viewed as compositions with parts in $\{1,2,\dots,m\}$. 
The difference now is that the size function of a word is defined to be its length, i.e., its number of entries, rather than the sum of its entries.
Setting $R(z)=z^1+z^2+\dots+z^m$ and $z=1$ in Proposition~\ref{prop:CR}, we obtain the following result.

\begin{corollary}
The generating function for $m$-ary words with respect to the degree of asymmetry is
$$W_m(t,x)=\frac{1+mx}{1-mx^2-m(m-1)tx^2}.$$
\end{corollary}

We derive that the number of $m$-ary words of length $n$ with degree of asymmetry $k$ is 
$$[t^kx^n]W_m(t,x)=m^{\lceil n/2 \rceil}\binom{\lfloor n/2 \rfloor}{k}(m-1)^k.$$
Thus, the random variable $X_n$, giving the degree of asymmetry of a uniformly random $m$-ary word of length $n$, follows a binomial distribution 
with parameters $\left\lfloor \frac{n}{2}\right\rfloor$ and $1-\frac{1}{m}$.
In particular, 
\begin{equation}\label{eq:limitW}
\E X_n= \left(1-\frac{1}{m}\right)\left\lfloor \frac{n}{2}\right\rfloor, \qquad  \V X_n=\frac{1}{m}\left(1-\frac{1}{m}\right)\left\lfloor \frac{n}{2}\right\rfloor,
\end{equation} 
and the limiting distribution is again normal.

\section{Restricted binary words}\label{sec:restricted}

Next we consider words over a binary alphabet $\{0,1\}$ with the additional restriction that the number of zeros equals the number of ones. 
Let $\G_n$ be the set of words consisting of $n$ zeros and $n$ ones. Clearly, $|\G_n|=\binom{2n}{n}$, and
$$\sum_{n\ge0}|\G_n| x^n=\frac{1}{\sqrt{1-4x}}.$$

\subsection{The degree of asymmetry over $\G_n$}

Let $[n]=\{1,2,\dots,n\}$. As in Section~\ref{sec:comp_words}, define the degree of asymmetry of $w=w_1w_2\dots w_{2n}\in\G_n$ as
$$
\da(w)=\left|\{i\in[n]:w_i\neq w_{2n+1-i}\}\right|,
$$
i.e., the number of symmetrically positioned pairs of entries where one is a $0$ and the other is a $1$.
Let $G(t,x)=\sum_{n\ge0}\sum_{w\in\G_n} t^{\da(w)} x^n$ be the corresponding generating function.

\begin{theorem}\label{thm:restricted}
$$G(t,x)=\frac{1}{\sqrt{1 - 4tx + 4 (t^2-1) x^2}}.$$
\end{theorem}

\begin{proof}
It is possible to prove this formula by using the bijection from~\cite{Elids} between $\G_n$ and so-called bicolored grand Motzkin paths, which can be shown to send the degree of asymmetry to the number of horizontal steps, and then enumerating such paths by the number of horizontal steps.

However, here we give a different, self-contained proof. For $w\in\G_n$ and a binary length-2 word $ab$, let 
$$
o_{ab}(w)=\left|\{i\in[n]:w_{2i-1}w_{2i}=ab\}\right|,
$$
that is, the number of occurrences of the consecutive subword $ab$ starting at an odd position. Applying the transformation 
\begin{equation}\label{eq:zigzag}
\begin{array}{rccc} Z: & \G_n & \to & \G_n \\ & w_1w_2\dots w_{2n} & \mapsto & w_1w_{2n}w_2 w_{2n-1}\dots w_n w_{n+1},\end{array}
\end{equation}
which reads $w$ in zig-zag, the statistic $\da$ becomes the statistic $o_{10}+o_{01}$.

To find an expression for $$G(t,x)=\sum_{n\ge0}\sum_{w\in\G_n} t^{o_{10}(w)+o_{01}(w)} x^n,$$ we decompose words $w\in\G_n$ uniquely as follows. Each pair $w_{2i-1}w_{2i}$, for $i\in[n]$, must be one of $00$, $01$, $10$ or $11$. 
The substring of $w$ consisting of pairs $w_{2i-1}w_{2i}\in\{00,11\}$ (we call these {\em twin pairs}) can be viewed itself as a word in $\G_m$, for some $m\le n/2$, where each letter has been doubled, i.e., each $0$ has been replaced with $00$ and each $1$ with $11$. This substring contributes 
\begin{equation}\label{eq:double}
\frac{1}{\sqrt{1-4x^2}}
\end{equation} 
to the generating function $G(t,x)$. Between each twin pair and the next, and also to the left of the first twin pair and to the right of the last one, $w$ contains (possibly empty) sequences of pairs $01$ and $10$ (we call these {\em non-twin pairs}). Each such sequence contributes $\frac{1}{1-2tx}$ to the generating function. Inserting these sequences after each twin pair corresponds to replacing $x$ with $\frac{x}{1-2tx}$ in equation~\eqref{eq:double}, and inserting a sequence before the first twin pair creates an extra factor $\frac{1}{1-2tx}$. It follows that
$$G(t,x)=\frac{1}{1-2tx}\frac{1}{\sqrt{1-4\left(\dfrac{x}{1-2tx}\right)^2}}=\frac{1}{\sqrt{(1-2tx)^2-4x^2}},$$
which simplifies to the expression stated above.
\end{proof}

The argument in the above proof also yields the more general multivariate generating function:
\begin{equation}\label{eq:4variables}
\sum_{n\ge0}\sum_{w\in\G_n} x_{00}^{o_{00}(w)}x_{11}^{o_{11}(w)}x_{10}^{o_{10}(w)}x_{01}^{o_{01}(w)}=\frac{1}{\sqrt{(1-x_{10}-x_{01})^2-4x_{00}x_{11}}}.
\end{equation}

Another consequence of the proof of Theorem~\ref{thm:restricted} is a closed formula for the number of words in $\G_n$ with degree of asymmetry $k$. Assuming that $n-k$ is even (since otherwise the number of such words is $0$), we have
$$
|\{w\in\G_n:\da(w)=k\}|=|\{w\in\G_n:o_{10}(w)+o_{01}(w)=k\}|=\binom{n}{k}\binom{n-k}{\frac{n-k}{2}}2^k.
$$
Indeed, words $w\in\G_n$ with $o_{10}(w)+o_{01}(w)=k$ consist of $n-k$ twin pairs and $k$ non-twin pairs. We can choose the positions of the non-twin pairs in $\binom{n}{k}$ ways, decide whether each such pair is a $10$ or a $01$ in $2^k$ ways, and finally choose the substring of twin pairs ---which consists of a word in $\G_{\frac{n-k}{2}}$
where each letter is doubled--- in $\binom{n-k}{\frac{n-k}{2}}$ ways.

\begin{corollary}\label{cor:limitGtx}
The sequence of random variables $X_n$ giving the degree of asymmetry of a uniformly random word in $\G_n$ converges to a normal distribution with expected value and variance given by
$$\E X_n=\frac{n}{2}+O(1),\qquad \V X_n=\frac{n}{4}+O(1). $$
\end{corollary}

\begin{proof}
Factoring the denominator in Theorem~\ref{thm:restricted}, we can write
$$G(t,x)=\frac{1}{\sqrt{(1-2(t-1)x)(1-2(t+1)x)}}.$$
In a neighborhood of $t=1$, the singularity closest to the origin is at $x=\rho(t)=\frac{1}{2(t+1)}$. Applying \cite[Thm.\ IX.12]{FS} and letting $f(t)=\rho(1)/\rho(t)=(t+1)/2$, it follows that $X_n$ converges to a normal distribution with mean $f'(1)n+O(1)=n/2+O(1)$ and variance $(f''(1)+f'(1)-f'(1)^2)n+O(1)=n/4+O(1)$.
The mean and variance can also be computed directly as in equation~\eqref{eq:EV}.
\end{proof}

Comparing Corollary~\ref{cor:limitGtx} with equation~\eqref{eq:limitW} for $m=2$, we see that the asymptotic behavior of the degree of asymmetry on words in $\G_n$ is the same as on binary words of length $2n$ without restrictions.

Words in $\G_n$ can be interpreted as grand Dyck paths, by replacing $0$ and $1$ with steps $U=(1,1)$ and $D=(1,-1)$, respectively. 
With this interpretation, one could consider a different concept of symmetry, where the height of symmetrically positioned steps 
in the grand Dyck path is taken into account. This notion is studied in~\cite{Elids}, along with the degree of symmetry of other lattice paths.
It is interesting to compare the normal limit law in Corollary~\ref{cor:limitGtx} with the analogous distribution for this other notion of symmetry, which 
in \cite{Elids} is shown to converge to a Rayleigh distribution with mean $\sqrt{\pi n}/2$.

\subsection{A surprising connection to patterns}\label{sec:bijection}

Another statistic on $\G_n$ that is closely related to $\da$ is the number of symmetrically positioned pairs of zeros. Specifically, for $w\in\G_n$, let 
$$\sz(w)=\left|\{i\in[n]:w_i=w_{2n+1-i}=0\}\right|.$$
For $w\in\G_n$, we have
\begin{equation}\label{eq:dasz}
\da(w) + 2 \sz(w) = n,
\end{equation}
since each pair of symmetrically positioned entries $w_iw_{2n+1-i}$ must be one of $01$, $10$, $00$ or $11$. The number of pairs $01$ and $10$ is counted by the statistic $\da(w)$, while the number of pairs $00$,  which must equal the number of pairs $11$ since $w$ has the same number of zeros and ones, is counted by $\sz(w)$. An equivalent perspective is obtained by applying the zigzag tranformation $Z$ from equation~\eqref{eq:zigzag}, which maps $\sz$ into the statistic $o_{00}$ counting occurrences of the consecutive subword $00$ in odd positions.
 It follows that 
\begin{equation}\label{eq:UUD}
\sum_{n\ge0}\sum_{w\in\G_n} u^{\sz(w)} x^n=\sum_{n\ge0}\sum_{w\in\G_n} u^{o_{00}(w)} x^n=\frac{1}{\sqrt{1 - 4x + 4 (1-u) x^2}},
\end{equation}
where the last equality is obtained by setting $x_{00}=ux$ and $x_{11}=x_{10}=x_{01}=x$ in equation~\eqref{eq:4variables}, or alternatively by
using equation~\eqref{eq:dasz} and noting that the substitution $G(u^{-1/2},u^{1/2}x)$ in Theorem~\ref{thm:restricted} yields the right-hand side of equation~\eqref{eq:UUD}. The coefficients of this generating function appear as table A051288 in~\cite{OEIS}.

Interestingly, there is a third, seemingly unrelated statistic on $\G_n$ that also has the same distribution. 
For $w\in\G_n$, let $\#_{001}(w)$ denote the number of occurrences of the consecutive subword $001$ in $w$. Occurrences will always refer to consecutive subwords even if not stated explicitly.

\begin{theorem}\label{thm:001}
For all $n,k$,
$$|\{w\in\G_n:\sz(w)=k\}|=|\{w\in\G_n:o_{00}(w)=k\}|=|\{w\in\G_n:\#_{001}(w)=k\}|.$$
\end{theorem}

\begin{proof}
By equation~\eqref{eq:UUD}, it is enough to show that the generating function for words in $\G_n$ where the variable $u$ marks the number of occurrences of $001$ coincides with the right-hand-side of this equation.

Let us first compute the generating function for words in $\G_n$ where an arbitrary subset of the occurrences of $001$ have been selected, or {\em marked}. Formally, we are enumerating pairs $(w,S)$, where $w\in\G_n$ and $S$ is a subset of the occurrences of $001$ in $w$. We introduce a variable $v$ that keeps track of $|S|$, that is, the number of marked occurrences of $001$. This can be done by starting with the generating function $\frac{1}{\sqrt{1-4yx}}$ for words in $\G_n$ with a variable $y$ keeping track of occurrences of $0$, and then making the substitution $y=1+vx$ to allow each $0$ in the word to either remain unchanged or be replaced with a marked occurrence of $001$. Thus, the generating function for pairs $(w,S)$, with a variable $v$ that keeps track of $|S|$, is 
\begin{equation}\label{eq:marked001}
 \frac{1}{\sqrt{1-4x-4vx^2}}.
\end{equation}

Finally, in order to obtain the generating function for words in $\G_n$ with a variable $u$ that keeps track of {\em all} occurrences of $001$, we make the substitution $v=u-1$ in equation~\eqref{eq:marked001}.
The reason for this substitution is that, if $T$ is the set of all occurrences of $001$ in a given word $w$, then
$$
\sum_{S\subseteq T} v^{|S|}=(v+1)^{|T|}, \quad  \text{and so}\quad \sum_{S\subseteq T} (u-1)^{|S|}=u^{|T|}.
$$
This substitution gives 
$$\sum_{n\ge0}\sum_{w\in\G_n} u^{\#_{001}(w)} x^n=\frac{1}{\sqrt{1 - 4x + 4 (1-u) x^2}}.$$
\end{proof}

\subsection{A bijective proof of Theorem~\ref{thm:001}}

The above proof of Theorem~\ref{thm:001} is relatively clean, but not bijective. In fact, the methods used to obtain the generating functions with respect to the statistics $o_{00}$ and $\#_{011}$ are quite different from each other, and so there is no clear way to turn our proof into a bijective one.

In this subsection we present another proof of Theorem~\ref{thm:001} by providing an explicit bijection that maps the number of occurrences of $00$ in odd positions to the number of occurrences of $001$. As a bonus property, it also maps the number of zeros in odd positions (we denote this statistic by $o_{0}(w)$) to the number of occurrences of $01$. Our bijection appears to be new, and it is not based on the above generating function argument.
Note that the equidistribution of the statistics $s_{0}$ and $o_{00}$ already has a bijective proof via the map $Z$ from equation~\eqref{eq:zigzag}.

\begin{theorem}\label{thm:bij}
There is an explicit bijection $\bij:\G_n\to\G_n$ such that, if $w\in\G_n$ and $w'=\bij(w)$, then
$$o_{00}(w)=\#_{001}(\bij(w)) \quad \text{and} \quad o_{0}(w)=\#_{01}(\bij(w)).$$
\end{theorem}

\begin{proof}
It will be convenient to describe the map $\bij:\G_n\to\G_n$ by passing through an intermediate set, denoted by $\Omega_n$, which consists of all 4-tuples $(Q,R,S,T)$ of pairwise disjoint (possibly empty) subsets of $[n]$ such that $Q\cup R\cup S\cup T=[n]$ and $|Q|=|R|$.

To each $w\in\G_n$, we associate $(Q,R,S,T)\in\Omega_n$ as follows. 
For $i\in[n]$, place $i$ in $Q$, $R$, $S$ or $T$ according to whether the pair $w_{2i-1}w_{2i}$ equals $00$, $11$, $01$, or $10$, respectively.
This map is clearly a bijection between $\G_n$ and $\Omega_n$. Additionally, if we let $k=o_{00}(w)$, then $|Q|=|R|=k$. 

Now write the elements of $Q$ as $q_1<q_2<\dots<q_k$, and the elements of $R$ as $r_1<r_2<\dots<r_k$. Set $r_0=0$ and $r_{k+1}=n+1$, and let $d_j=|\{i\in T: r_{j-1}<i<r_j\}$ for all $j\in[k+1]$. Finally, define $$\bij(w)=\mu_1 1 \mu_2 1 \dots \mu_n 1 \mu_{n+1},$$ where $\mu_{q_j}=0^{2+d_j}$ for all $j\in[k]$, $\mu_{n+1}=0^{d_{k+1}}$, $\mu_i=0$ for all $i\in S$, and $\mu_i$ is empty for all $i\in R\cup T$.

The resulting word $\bij(w)$ is in $\G_n$, since its number of zeros is 
$$(2+d_1)+(2+d_2)+\dots+(2+d_k)+d_{k+1}+|S|=2k+|T|+|S|=|Q|+|R|+|T|+|S|=n.$$
The number of occurrences of $001$ in $\bij(w)$ is precisely $|Q|=k=o_{00}(w)$.
Additionally, the number of occurrences of $01$ in $\bij(w)$ equals $|Q|+|S|$, which is the number of indices $i\in[n]$ such that $w_{2i-1}=0$, namely $o_0(w)$.

To see that $\bij$ is a bijection, let us show how to reconstruct the 4-tuple $(Q,R,S,T)\in\Omega_n$ given a word $\mu_1 1 \mu_2 1 \dots \mu_n 1 \mu_{n+1}\in\G_n$. For $i\in[n]$, clearly $i\in S$ if and only if $\mu_i=0$, and $i\in Q$ if and only if $\mu_i$ has length at least $2$. The lengths of the words $\mu_i$ for $i\in Q$ can be used to recover the sequence $d_j$, which then determines how to place the elements of $[n]\setminus(Q\cup S)$ into the two disjoint subsets $R$ and $T$.
\end{proof}

\begin{example}\label{ex:bij}
If $w=\textcolor{blue}{01}\,\textcolor{violet}{11}\,\textcolor{red}{00}\,\textcolor{red}{00}\,\textcolor{teal}{10}\,\textcolor{violet}{11}\,\textcolor{blue}{01}\,\textcolor{teal}{10}\,\textcolor{violet}{11}\,\textcolor{red}{00}\,\textcolor{blue}{01}\in\G_{11}$, then $Q=\{\textcolor{red}{3},\textcolor{red}{4},\textcolor{red}{10}\}$, $R=\{\textcolor{violet}{2},\textcolor{violet}{6},\textcolor{violet}{9}\}$, $S=\{\textcolor{blue}{1},\textcolor{blue}{7},\textcolor{blue}{11}\}$, and $T=\{\textcolor{teal}{5},\textcolor{teal}{8}\}$.
We have $d_1=0$, $d_2=1$, $d_3=1$, $d_4=0$, and so $\mu_{\textcolor{red}{3}}=\textcolor{red}{00}$, $\mu_{\textcolor{red}{4}}=\textcolor{red}{000}$, and $\mu_{\textcolor{red}{10}}=\textcolor{red}{000}$. Thus,
$$\bij(w)=\textcolor{blue}{0}\, 1\, 1\, \textcolor{red}{00}\, 1\, \textcolor{red}{000}\, 1\, 1\, \textcolor{blue}{0}\, 1\, 1\, 1\, \textcolor{red}{000}\, 1\, \textcolor{blue}{0}\, 1.$$
This example is illustrated in Figure~\ref{fig:bij}
\end{example}

\newcommand\N{-- ++(0,1)}
\newcommand\EE{-- ++(1,0)}
\newcommand\U{-- ++(.5,.5) circle (0.05)}
\newcommand\D{-- ++(.5,-.5) circle (0.05)}
\newcommand\RR{++(.5,0)}
\newcommand\dS{\filldraw[ultra thick,blue,dotted] (c)\U\D coordinate (c);}
\newcommand\dR{\filldraw[ultra thick,violet] (c)\D\D coordinate (c);}
\newcommand\dQ{\filldraw[ultra thick,red,dashed] (c)\U\U coordinate (c);}
\newcommand\dT{\filldraw[ultra thick,teal] (c)\D\U coordinate (c);}
\newcommand\lS{\draw[blue] (c)\RR node{$S$}\RR coordinate (c);}
\newcommand\lR{\draw[violet] (c)\RR node{$R$}\RR coordinate (c);}
\newcommand\lQ{\draw[red] (c)\RR node{$Q$}\RR coordinate (c);}
\newcommand\lT{\draw[teal] (c)\RR node{$T$}\RR coordinate (c);}
\newcommand\sS{\draw[ultra thick,blue,dotted] (c)\N coordinate (c);
\draw[ultra thick] (c)\EE coordinate (c);}
\newcommand\sR{\fill[thick,violet] (c) circle (.1);
\draw[ultra thick] (c)\EE coordinate (c);}
\newcommand\sQ[1]{
\foreach \i in {1,...,#1}{ \draw[ultra thick,red,dashed] (c)\N coordinate (c);}
\draw[ultra thick] (c)\EE coordinate (c);}
\newcommand\sT{\fill[thick,teal] (c) circle (.1);
\draw[ultra thick] (c)\EE coordinate (c);}

\begin{figure}[htb]
\centering
\begin{tikzpicture}[scale=.9]

\coordinate (c) at (0,12.7);
\dS\dR\dQ\dQ\dT\dR\dS\dT\dR\dQ\dS

\coordinate (c) at (-.5,11.3);
\lS\lR\lQ\lQ\lT\lR\lS\lT\lR\lQ\lS

\draw[ultra thin,lightgray] (0,0) grid (11,11);

\coordinate (c) at (0,0);
\sS\sR\sQ{2}\sQ{3}\sT\sR\sS\sT\sR\sQ{3}\sS

\draw [violet,decorate,decoration={brace,amplitude=3pt},yshift=-2pt]
(.9,11) -- (-.3,11) node [midway,below,teal] {$d_1=0$};
\draw [violet,decorate,decoration={brace,amplitude=3pt},yshift=-2pt]
(4.9,11) -- (1.1,11) node [midway,below,teal] {$d_2=1$};
\draw [violet,decorate,decoration={brace,amplitude=3pt},yshift=-2pt]
(7.9,11) -- (5.1,11) node [midway,below,teal] {$d_3=1$};
\draw [violet,decorate,decoration={brace,amplitude=3pt},yshift=-2pt]
(10.3,11) -- (8.1,11) node [midway,below,teal] {$d_4=0$};

\draw [decorate,decoration={brace,amplitude=3pt},xshift=-3pt,yshift=0pt]
(2,1) -- (2,3) node [midway,left] {$2+\textcolor{teal}{d_1}$};
\draw [decorate,decoration={brace,amplitude=3pt},xshift=-3pt,yshift=0pt]
(3,3) -- (3,6) node [midway,left] {$2+\textcolor{teal}{d_2}$};
\draw [decorate,decoration={brace,amplitude=3pt},xshift=-3pt,yshift=0pt]
(9,7) -- (9,10) node [midway,left] {$2+\textcolor{teal}{d_3}$};

\end{tikzpicture}
\caption{An illustration of the bijection $\bij$ from Theorem~\ref{thm:bij} applied to the word $w$ from Example~\ref{ex:bij}. Zeros and ones in $w$ are represented with up and down steps in the top path. Zeros and ones in $\bij(w)$ are represented by north and east steps in the bottom path.}
\label{fig:bij}
\end{figure}
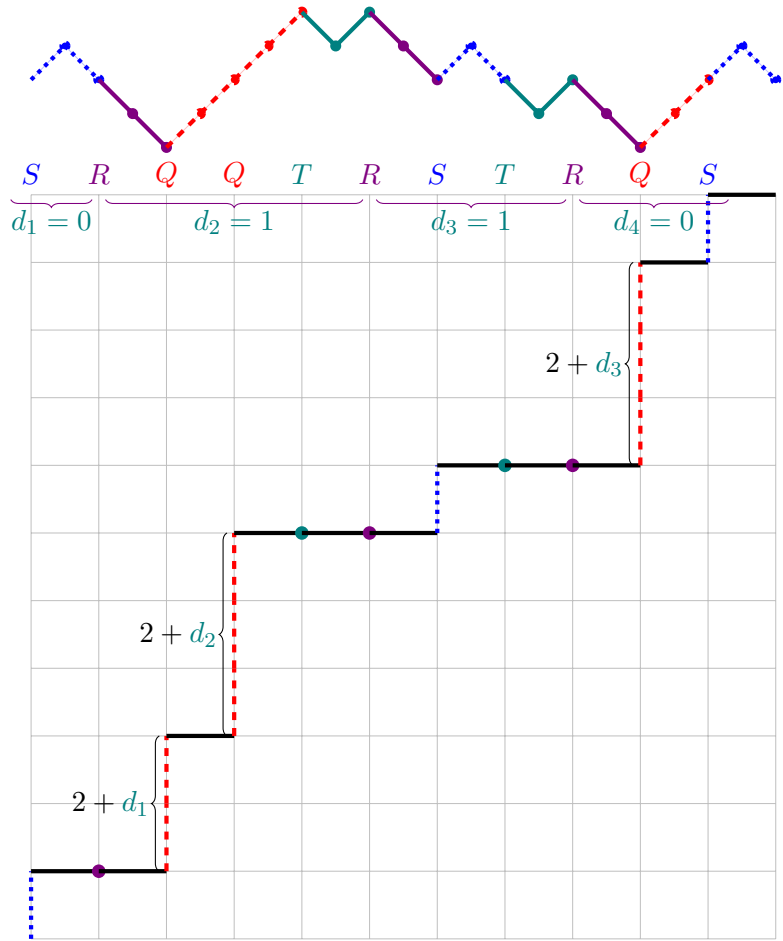

We conclude this section with a simple formula for the cardinality of the sets in Theorem~\ref{thm:001}. Compare this with the cumbersome triple-summation formula that appears  in~\cite[Cor.\ 3.3]{MSTT} for the number of words in $\G_n$ with a given number of occurrences of $001$. 

\begin{corollary}
The number of words in $\G_n$ with $k$ occurrences of $001$ equals
$$\binom{n}{2k}\binom{2k}{k}2^{n-2k}.$$
The number of words in $\G_n$ with $k$ occurrences of $001$ and $\ell$ occurrences of $01$ equals
$$\binom{n}{2k}\binom{2k}{k}\binom{n-2k}{\ell-k}=\frac{n!}{k!^2(\ell-k)!(n-k-\ell)! }.$$
\end{corollary}

\begin{proof}
By Theorem~\ref{thm:bij}, the first problem is equivalent to counting 4-tuples $(Q,R,S,T)\in\Omega_n$ with $|Q|=|R|=k$ (or equivalently, words 
$w\in\G_n$ with $o_{00}(w)=k$). There are $\binom{n}{2k}$ ways to choose the elements in $Q\cup R$, $\binom{2k}{k}$ ways to separate them into two sets $Q$ and $R$, and  $2^{n-2k}$ ways to place the remaining $n-2k$ elements into $S$ or $T$.
If, additionally, we require $|S|=\ell-k$, then the number of ways to choose the elements in $S$ is $\binom{n-2k}{\ell-k}$.
\end{proof}

By Theorem~\ref{thm:bij}, the generating function for words in $\G_n$ with variables $u$ and $v$ keeping track of the number of occurrences of $001$ and $01$, respectively,  is obtained by setting $x_{10}=x$, $x_{01}=vx$,  $x_{00}=uvx$,  $x_{11}=x$ in equation~\eqref{eq:4variables}:
$$\sum_{n\ge0}\sum_{w\in\G_n} u^{\#_{001}(w)} v^{\#_{01}(w)} x^n=\sum_{n\ge0}\sum_{w\in\G_n} u^{o_{00}(w)} v^{o_{0}(w)} x^n=\frac{1}{\sqrt{(1 - x-vx)^2 - 4 uv x^2}}.$$

\section{Binary trees}\label{sec:trees}

A {\em binary tree} is a rooted tree where each node has at most two children, referred to as the left child and the right child.
It is well known that the number of binary trees with $n$ nodes is the $n$th Catalan number, and so the generating function for binary trees where $x$ marks the number of nodes is 
\begin{equation}
\label{eq:Cat}
\Cat(x)=\frac{1-\sqrt{1-4x}}{2x}.
\end{equation}

The position of each node in a binary tree $T$ can be specified by the word over $\{L,R\}$ that describes the sequence of left and right children in the path from the root to that node. We call this word the {\em address} of the node, and define its {\em complement} to be the word obtained by switching $L$s with $R$s. 
We say that two nodes in $T$ are {\em mirror images} if their addresses are complements of each other.
Define the {\em degree of asymmetry} of a binary tree, denoted by $\da(T)$, to be the number of nodes that are not mirror images of any node in the tree. See Figure~\ref{fig:binarytree} for an example.

\begin{figure}[htb]
\centering
 \begin{tikzpicture}[scale=1]
     \draw[thick] (0,0) coordinate(d0) -- (-2,-1) coordinate(d1) -- (-3,-2) coordinate(d2) -- (-2.5,-3) coordinate(d6);
     \draw[thick] (d1) -- (-1,-2) coordinate(d3) -- (-1.5,-3) coordinate(d4); 
      \draw[thick] (d3) -- (-.5,-3) coordinate(d5); 
           \draw[thick] (d0) -- (2,-1) coordinate(d7) -- (3,-2) coordinate(d8) -- (3.5,-3) coordinate(d9) -- (3.25,-4) coordinate(d11);
             \draw[thick] (d8) -- (2.5,-3) coordinate(d10); 
      \foreach \x in {0,...,11} {
        \draw[fill] (d\x) circle (2pt);
      }
       \foreach \x in {3,4,5,9,11} {
        \draw[red,thick] (d\x) circle (4pt);
      }
        \end{tikzpicture}
      \caption{A binary tree $T$ with $\da(T)=5$. The 5 nodes without mirror images are circled in red.}
      \label{fig:binarytree}
\end{figure}

Let $B(t,x)$ be the ordinary generating function for binary trees where $t$ marks the degree of asymmetry.

\begin{theorem}\label{thm:binary}
$$B(t,x)=1+\frac{1-\sqrt{1+4x\left(x-\frac{1-\sqrt{1-4tx}}{t}\right)}}{2x}$$
\end{theorem}

\begin{proof}
Let $D(t,x)$ be the generating function for pairs of binary trees, where $x$ marks the total number of nodes, and $t$ marks the number of nodes on either tree for which there is no node in the other tree with the same address; let us call these {\em unique} nodes. 

To find an equation for $D(t,x)$, consider three possibilities for a pair of trees: both trees are empty, one tree is empty but the other is not, or both trees are non-empty.
When one tree is empty but the other is not, the latter one contributes a term $\Cat(tx)-1$, since all of its nodes are unique. When both trees are non-empty, they contribute $x^2D(t,x)^2$, since the left subtrees of the two roots form themselves a pair of binary trees where the variable $t$ marks the number of unique nodes, and so do the right subtrees of the two roots. It follows that
$$D(t,x)=1+2(\Cat(tx)-1)+x^2D(t,x)^2,$$
and solving for $D(t,x)$ we get
\begin{equation}\label{eq:D}
D(t,x)=\frac{1-\sqrt{1+4x^2(1-2\Cat(tx))}}{2x^2}.
\end{equation}

Next we relate $B(t,x)$ and $D(t,x)$. Given a non-empty binary tree, consider the pair of trees consisting of the left subtree of the root and the reflection (over a vertical line) of the right subtree of the root. Pairs of nodes that were mirror images in the original tree become pairs of nodes with the same address in the resulting pair of trees. Thus, the degree of asymmetry of the original tree equals the number of unique nodes in the resulting pair of trees. It follows that $B(t,x)=1+xD(t,x)$. Using equations~\eqref{eq:D} and~\eqref{eq:Cat}, we obtain the stated expression for $B(t,x)$.
\end{proof}

By construction, $B(1,x)=C(x)$. While this is not immediately clear from the expression in Theorem~\ref{thm:binary}, it can be checked using the identity
$$\sqrt{1-4x+4x^2+4x\sqrt{1-4tx}}=\sqrt{(2x+\sqrt{1-4x})^2}=2x+\sqrt{1-4x}.$$
On the other hand, setting $t=0$ in $B(t,x)$, we obtain the generating function for symmetric binary trees
$$B(0,x)=1+\frac{1-\sqrt{1-4x^2}}{2x}=1+x\Cat(x^2).$$

To study the asymptotic behavior of the degree of asymmetry, we will consider a complementary statistic, namely the number of pairs of nodes (not counting the root) that are mirror images of each other. This statistic, which we call the {\em degree of symmetry}, has a more natural limiting distribution, as we will see next. If $T$ is a binary tree with $n$ nodes and
$\ds(T)$ denotes its degree of symmetry, then $2\ds(T)+\da(T)+1=n$. 
Let $\overline{B}(s,x)$ be the generating function for binary trees where the variable $s$ marks the degree of symmetry. A slight modification of the proof of Theorem~\ref{thm:binary} shows that $\overline{B}(s,x)=1+x \overline{D}(s,x)$, where $\overline{D}(s,x)=1+2(\Cat(x)-1)+sx^2\overline{D}(s,x)^2$, from where
\begin{equation}\label{eq:oB}
\overline{B}(s,x)=1+\frac{1-\sqrt{1+4sx\left(x-1+\sqrt{1-4x}\right)}}{2sx}.
\end{equation}
When $s=0$, we get
$$\overline{B}(0,x)=1-x+2xC(x)=2-x-\sqrt{1-4x},$$
whose coefficients for $n\ge2$ are doubled Catalan numbers. Indeed, binary trees $T$ with $\ds(T)=0$ are those where one of the two subtrees of the root is empty.

\begin{corollary}\label{cor:limitBsx}
The sequence of random variables $X_n$, giving the degree of symmetry of a uniformly random binary tree with $n$ nodes, converges to a discrete limit law with
\begin{equation}\label{eq:discretelimit}
\lim_{n\to\infty}\Pr(X_n=k)=\frac{1}{2}\binom{2k}{k}\left(\frac{3}{16}\right)^k
\end{equation}
for all $k\ge0$.
Expected value and variance are given by
$$\E X_n=\frac{3}{2}+o(1),\qquad \V X_n=6+o(1). $$
\end{corollary}

\begin{proof}
We can write equation~\eqref{eq:oB} as
$$\overline{B}(s,x)=1+\frac{1+F(s,x)}{2sx}, \quad \text{where } F(s,x)=-\sqrt{1+4sx\left(x-1+\sqrt{1-4x}\right)}.$$
This has the advantage that the generating function $F(s,x)$ can be expressed as a composition $F(s,x)=g(s h(x))$, where
$$g(x)=-\sqrt{1-4x} \quad \text{and} \quad h(x)=x\left(x-1+\sqrt{1-4x}\right);$$
therefore it is amenable to the singularity analysis techniques in \cite[Chapter IX]{FS}.

The radii of convergence of $g$ and $h$ are $\rho_g=\rho_h=1/4$. Letting $\tau_h=h(\rho_h)=3/16$, we have $\tau_h<\rho_g$, which corresponds to the {\em subcritical composition} case
, and so we can apply \cite[Prop. IX.1]{FS} to obtain the discrete limit law. Specifically, for fixed $k\ge1$,
$$\lim_{n\to\infty} \frac{[s^kx^n]F(s,x)}{[x^n]F(1,x)}=\frac{k\tau_h^{k-1}[x^k]g(x)}{g'(\tau_h)}=\frac{1}{2}\binom{2k-2}{k-1}\left(\frac{3}{16}\right)^{k-1}.$$
from where equation~\eqref{eq:discretelimit} follows, noting that
$$\Pr(X_n=k)= \frac{[s^kx^n]\overline{B}(s,x)}{[x^n]\overline{B}(1,x)}=\frac{[s^{k+1}x^{n+1}]F(s,x)}{[x^{n+1}]F(1,x)}
$$
for $n\ge1$.

The mean and variance can now easily be computed from equation~\eqref{eq:discretelimit}, or directly from $\overline{B}(s,x)$ as in equation~\eqref{eq:EV}.
\end{proof}

\section{Matchings}\label{sec:matchings}

Let $\M_n$ denote the set of perfect matchings (which we will call simply {\em matchings} for short) of $[2n]$.
It is well known that 
$$|\M_n|=(2n-1)!!=(2n-1)\cdot(2n-3)\cdot\cdots\cdot3\cdot1=\frac{(2n)!}{2^n n!},$$
and so the exponential generating function (EGF) for matchings is
$$\sum_{n\ge0} |\M_n| \frac{x^n}{n!}=\sum_{n\ge0} \frac{(2n)!}{2^n n!} \frac{x^n}{n!} = \sum_{n\ge0} \binom{2n}{n} \frac{x^n}{2^n}=\frac{1}{\sqrt{1-2x}}.$$

Matchings are often drawn by placing $2n$ vertices on a line, labeled $1,2,\dots,2n$ from left to right, and placing an arc connecting each pair of matched vertices. See Figure~\ref{fig:matching} for an example.

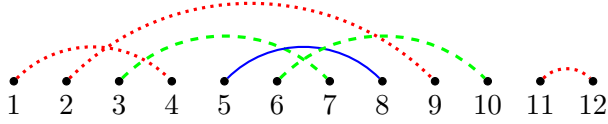
\begin{figure}[htb]
\centering
  \begin{tikzpicture}[scale=0.7]
   \foreach \i in {1,...,12}
        \node[pnt,label=below:$\i$] at (\i,0)(\i) {};
   \draw[very thick,red,dotted] (1)  to [bend left=45] (4);
   \draw[very thick,red,dotted](2)  to [bend left=45] (9);
   \draw[very thick,green,dashed](3)  to [bend left=45] (7);
   \draw[thick, blue] (5)  to [bend left=45] (8);
   \draw[very thick,green,dashed](6)  to [bend left=45] (10);
   \draw[very thick,red,dotted](11)  to [bend left=45] (12);
   \end{tikzpicture}
\caption{The matching $M=\{\{1,4\},\{2,9\},\{3,7\},\{5,8\},\{6,10\},\{11,12\}\}\in\M_{12}$ has one centered arc (blue, solid), one pair of coupled arcs (green, dashed), and three asymmetric arcs (red, dotted).}
\label{fig:matching}
\end{figure}

For a subset $A\subseteq[n]$, let $\overline{A}=\{2n+1-i:i\in A\}$. In a matching $M\in\M_n$, each arc $\{i,j\}\in M$ (that is, a pair of matched vertices) is of one of three types: 
\begin{itemize}
\item {\em centered}, if $j=2n+1-i$ (that is, $\overline{\{i,j\}}=\{i,j\}$);
\item {\em coupled}, if it is not centered but $\overline{\{i,j\}}\in M$ (in this case, we call $\{i,j\}$ and $\overline{\{i,j\}}$ a pair of coupled arcs);
\item {\em asymmetric}, if it is neither centered nor coupled. 
\end{itemize}

We define the {\em degree of asymmetry} of $M\in\M_n$, denoted by $\da(M)$, to be the number of asymmetric arcs that it contains.
Similarly, denote by $c(M)$ the number of centered arcs, and by $p(M)$ the number of pairs of coupled arcs in $M$. Note that 
\begin{equation}\label{eq:cpda}
c(M)+2p(M)+\da(M)=n.
\end{equation}
For example, the matching in Figure~\ref{fig:matching} has $c(M)=1$, $p(M)=1$, and $\da(M)=3$.

\begin{theorem}\label{thm:matchings}
The EGF for matchings with respect to the number of centered arcs, the number of pairs of coupled arcs, and the degree of asymmetry, is
$$H(r,s,t,x)=\sum_{n\ge0} \sum_{M\in\M_n} r^{c(M)}s^{p(M)}t^{\da(M)} \frac{x^n}{n!}=\frac{e^{(r-t)x+(s-t^2)x^2}}{\sqrt{1-2tx}}.$$
\end{theorem}

\begin{proof}
The EGF for matchings all of whose arcs are centered is $e^x$, whereas the EGF for matchings all of whose arcs are coupled is $e^{x^2}$.
More generally, the EGF for matchings all of whose arcs are centered or coupled, with variables $u$ and $v$ keeping track of the number of centered and the number of pairs of coupled arcs, respectively, is $e^{ux+vx^2}$. Indeed, such matchings on $2n$ vertices are equivalent to 
set partitions of $[n]$ into blocks of size one (marked by $u$) and two (marked by $v$).

Next, we consider matchings in $\M_n$ where an arbitrary subset of the centered arcs and an arbitrary subset of the pairs of coupled arcs have been marked, using a similar technique as in our first proof of Theorem~\ref{thm:001}. We denote the set of such marked matchings by $\M^\ast_n$. Formally, $\M^\ast_n$ is the set of triples $(M,S_1,S_2)$ where $M\in\M_n$, $S_1$ is a subset of centered arcs of $M$, and $S_2$ is a subset of pairs of coupled arcs of $M$. The EGF for marked matchings is
$$H^\ast(u,v,x)=\sum_{n\ge0}\sum_{(M,S_1,S_2)\in\M^\ast_n} u^{|S_1|}v^{|S_2|} \frac{x^n}{n!}=\frac{e^{ux+vx^2}}{\sqrt{1-2x}}.$$
Indeed, marked matchings are obtained by partitioning $[n]$ into three (possibly empty) sets, say $[n]=A_0\sqcup A_1\sqcup A_2$, placing an arbitrary matching with no marked arcs on the vertices $A_0\cup \overline{A_0}$, a matching consisting of only centered arcs on the vertices $A_1\cup \overline{A_1}$, and a matching consisting of only coupled arcs on the vertices $A_2\cup \overline{A_2}$.

Using again the fact that, for any finite set $T$,
$$
\sum_{S\subseteq T} u^{|S|}=(u+1)^{|T|}, \quad  \text{and so}\quad \sum_{S\subseteq T} (t-1)^{|S|}=t^{|T|},
$$
it follows that the EGF for matchings with respect to the statistics $c$ and $p$ is
\begin{equation}\label{eq:H}
H(r,s,1,x)=\sum_{n\ge0}\sum_{M\in\M_n} r^{c(M)}s^{p(M)} \frac{x^n}{n!}=H^\ast(r-1,s-1,x)=\frac{e^{(r-1)x+(s-1)x^2}}{\sqrt{1-2x}}.
\end{equation}

Finally, using equation~\eqref{eq:cpda}, we deduce that 
\[
H(r,s,t,x)=H(r/t,s/t^2,1,tx)=\frac{e^{(r-t)x+(s-t^2)x^2}}{\sqrt{1-2tx}}.\qedhere
\]
\end{proof}

Some specializations of the formula in Theorem~\ref{thm:matchings} are known.
For example, $H(1,1,0,x)=e^{x+x^2}$ is the EGF for symmetric matchings, that is, those that are fixed by the transformation $i\mapsto 2n+1-i$ applied to the elements of $[2n]$. Its coefficients give sequence A047974 in~\cite{OEIS}.
Similarly, the specialization $$H(0,1,1,x)=\frac{e^{-x}}{\sqrt{1-2x}}$$ is the EGF for matchings with no centered arcs. Its coefficients give sequence A053871 in~\cite{OEIS}.

Other specializations do not appear in~\cite{OEIS} yet. For example, $$H(0,0,1,x)=\frac{e^{-x-x^2}}{\sqrt{1-2x}}$$ is the EGF
for completely asymmetric matchings, that is, those containing no centered or coupled arcs. The first terms of the sequence starting at $n=1$ are $0$, $0$, $8$, $48$, $384$, $4480$, $59520$, $897792$, $15368192$.
Similarly, $$H(1,0,1,x)=\frac{e^{-x^2}}{\sqrt{1-2x}}$$ is the EGF
for matchings with no coupled arcs. The first terms of the sequence starting at $n=1$ are 
$1$, $1$, $9$, $81$, $705$, $7665$, $100905$, $1524705$, $26022465$.

Next we describe the asymptotic behavior of the number of arcs of each type. Because of equation~\eqref{eq:cpda}, it is sufficient to focus on the statistics $c$ and $p$, whose distributions converge to very natural discrete laws.

\begin{corollary}\label{cor:limitH}
Let $X_n$ and $Y_n$ be sequences of random variables giving the number of centered arcs and the number of pairs of coupled arcs, respectively, in a uniformly random matching in $\M_n$. Then $X_n$ converges to a Poisson law with parameter $\lambda=1/2$, and $Y_n$ converges to a Poisson law with parameter $\lambda=1/4$, i.e.
$$
\lim_{n\to\infty}\Pr(X_n=k)=\frac{e^{-\frac12}}{2^k k!} \quad\text{and}\quad \lim_{n\to\infty}\Pr(Y_n=\ell)=\frac{e^{-\frac14}}{4^\ell \ell!}
$$
for all $k,\ell\ge0$.
In addition, $X_n$ and $Y_n$ are asymptotically independent, in the sense that
$$
\lim_{n\to\infty}\Pr(X_n=k,Y_n=\ell)=\frac{e^{-\frac34}}{2^k 4^\ell  k!\ell!}.
$$
\end{corollary}

\begin{proof}
To compute these probabilities, we extract the relevant coefficients of the EGF in Theorem~\ref{thm:matchings}, and then determine their behavior as $n\to\infty$ using standard singularity analysis \cite[Thm.\ VI.1]{FS}, noting that the dominant singularity is at $x=1/2$:
$$[r^kx^n]H(r,1,1,x)=[x^n]\frac{e^{-x}}{\sqrt{1-2x}}\frac{x^k}{k!}\sim \frac{e^{-\frac12}}{2^k k!}\frac{2^n}{\sqrt{\pi n}}.$$
Dividing by $$[x^n]H(1,1,1,x)=[x^n]\frac{1}{\sqrt{1-2x}}\sim \frac{2^n}{\sqrt{\pi n}}$$ and using equation~\eqref{eq:Xn}, which also applies  to exponential generating functions, we obtain the stated expression for $\lim_{n\to\infty}\Pr(X_n=k)$.

Similarly, the expression for $\lim_{n\to\infty}\Pr(Y_n=\ell)$ follows from the fact that
$$[s^\ell x^n]H(1,s,1,x)=[x^n]\frac{e^{-x^2}}{\sqrt{1-2x}}\frac{x^{2\ell}}{\ell!}\sim \frac{e^{-\frac14}}{4^\ell \ell!}\frac{2^n}{\sqrt{\pi n}}.$$
Finally, the joint distribution is obtained by computing 
$$[r^k s^\ell x^n]H(r,s,1,x)=[x^n]\frac{e^{-x-x^2}}{\sqrt{1-2x}}\frac{x^k}{k!}\frac{x^{2\ell}}{\ell!}\sim \frac{e^{-\frac34}}{2^k 4^\ell k! \ell!}\frac{2^n}{\sqrt{\pi n}}$$
and using that
\[
\Pr(X_n=k,Y_n=\ell)=\frac{[r^k s^\ell x^n]H(r,s,1,x)}{[x^n]H(1,1,1,x)}.
\qedhere
\]
\end{proof}

\section{Permutations}\label{sec:perm}

One way to measure symmetry in a permutation $\pi\in\S_n$ is by considering how close the configuration of dots in coordinates $(i,\pi(i))$, for $i\in[n]$, is to being symmetric with respect to the diagonal. Under this measure, it is natural to define the degree of asymmetry of a permutation as the number of dots whose reflection along the diagonal does not match another dot, that is,
$$\da(\pi):=\left|\{i:\pi(i)\neq\pi^{-1}(i)\}\right|.$$

Note that $\da(\pi)$ equals the number of elements in $[n]$ that do not belong to $1$-cycles or $2$-cycles of $\pi$.
Using the standard decomposition of permutations as products of disjoint cycles (see e.g. \cite[Ex.\ II.12]{FS}), and denoting by $c_m(\pi)$ the number of $m$-cycles of $\pi$, we obtain the EGF
$$P(r,s,t,x)=\sum_{n\ge0}\sum_{\pi\in\S_n} r^{c_1(\pi)} s^{c_2(\pi)} t^{\da(\pi)} \frac{x^n}{n!}=\frac{e^{(r-t)x+(s-t^2)x^2/2}}{1-tx}.$$

The specialization $P(1,1,0,x)=e^{x+x^2/2}$ is the EGF for involutions \cite[A000085]{OEIS}, which are those permutations that are symmetric according to our definition.
On the other hand, the specialization $$P(0,0,1,x)=\frac{e^{-x-x^2/2}}{1-x}$$ is the EGF for permutations with no $1$-cycles or $2$-cycles \cite[A038205]{OEIS}, that is, permutations that are completely asymmetric.

It is well known that, for any fixed $m\ge1$, the number of $m$-cycles in a uniformly random permutation in $\S_n$ converges to a Poisson law with parameter $\lambda=1/m$ \cite[Example IX.4]{FS}. Additionally, as in the proof of Corollary~\ref{cor:limitH}, the distributions of these statistics for different values of $m$ are asymptotically independent.

\section{A problem regarding unimodal compositions}

In Section~\ref{sec:comp}, we found the generating function for compositions with respect to the degree of asymmetry. 
Let us consider a variant of this problem by restricting to unimodal compositions, namely, compositions $(a_1,a_2,\dots,a_m)$ for which there exists $i$ such that 
\begin{equation}\label{eq:unimodal} a_1\le a_2\le \dots\le a_i\ge a_{i+1}\ge\dots \ge a_m. 
\end{equation}

\begin{problem}\label{prob:unimodal}
Find the generating function $U(t,z)$ for unimodal compositions where $t$ marks the degree of asymmetry and $z$ marks the sum of the parts.
\end{problem}

Let us show that the specializations $U(1,z)$ and $U(0,z)$ are easy to obtain, and their coefficients give sequences A001523 and A096441 in~\cite{OEIS}, respectively.
Recall that a partition is a weakly decreasing composition, and that the generating function for partitions with parts of size less than $k$, where $z$ marks the sum of the parts, is 
\begin{equation}\label{eq:partitions} \prod_{i=1}^{k-1}\frac{1}{1-z^i}. 
\end{equation}

It follows that the generating function for unimodal compositions is
$$U(1,z)=\sum_{k\ge1} \frac{z^k}{1-z^k}\left(\prod_{i=1}^{k-1}\frac{1}{1-z^i}\right)^2.$$
Indeed, let $k$ be the largest part of a unimodal composition. The factor $\frac{z^k}{1-z^k}$ accounts for copies of $k$, of which there is at least one. To the left (respectively, right) of $k$ there is a weakly increasing (respectively, decreasing) sequence of entries less than $k$; equivalently, a partition with parts of size less than $k$, contributing a factor as in~\eqref{eq:partitions}.
It is also possible, as shown in~\cite[Cor.\ 2.5.3]{EC1}, to give the following product formula for $U(1,z)$ using a sieve process:
$$U(1,z)=\left(\sum_{k\ge1} (-1)^{k-1} z^{\binom{k+1}{2}}\right)\prod_{k\ge1} (1-x^k)^{-2}.$$

A similar argument can be used to count unimodal compositions with degree of asymmetry equal to 0, that is, palindromic unimodal compositions:
$$U(0,z)=\sum_{k\ge1} \frac{z^k}{1-z^k}\prod_{i=1}^{k-1}\frac{1}{1-z^{2i}}.$$
In this case, the weakly increasing piece to the left of the largest part has to be a mirror image of the weakly decreasing piece to the right, and so these two pieces correspond to the same partition with parts of size less than $k$. 

We can obtain another expression for $U(0,z)$ by describing a bijection between palindromic unimodal compositions and partitions where all the parts are odd or all the parts are even. Given a palindromic unimodal composition $a=(a_1,a_2,\dots,a_m)$, construct a partition $\lambda$ by letting $\lambda_i=|\{j:a_j\ge i\}|$. If $m$ is even (odd), then $\lambda_i$ is even (odd) for all $i$, because $a$ is palindromic and unimodal. Additionally, given a partition $\lambda$ where all parts are odd or all parts are even, we can recover $a$ as follows: start with the composition $a_j=0$ for all $1\le j\le m$. For each part $a_i$, add $1$ to $a_j$ if $j$ is one of the $a_i$ most central indices, that is, those closest to $(m+1)/2$.
It follows that
$$U(0,z)=\prod_{j\ge1}\frac{1}{1-z^{2j}}+\prod_{j\ge1}\frac{1}{1-z^{2j-1}}-2.$$

As a variation of Problem~\ref{prob:unimodal}, one can consider unimodal compositions with the additional restriction that their maximum lies in the middle, that is, those satisfying equation~\eqref{eq:unimodal} with $i=(k+1)/2$ when $k$ is odd, and with $i\in\{k/2,k/2+1\}$ when $k$ is even. 
A different version of this question, where the size of a composition is defined to be the largest part plus the number of parts (equivalently, the semiperimeter of the associated unimodal bargraph), is addressed in~\cite{Elids}.

\end{document}